\documentclass{amsart}
\usepackage{fullpage}
\usepackage{all2021}
\usepackage{algorithm}
\usepackage{algorithmic}

\newtheorem{conjecture}[thm]{Conjecture}

\newcommand{\R}{\mathbb{R}}

\newcommand{\rank}{R}

\newcommand{\ignore}[1]{}

\newcommand{\clos}{\operatorname{cl}}
\newcommand{\inter}{\operatorname{int}}

\usepackage[T1]{fontenc}

\usepackage{soul}

\usepackage{tikz}

\usepackage{xcolor}

\begin{document}
	
	\title{Real subrank of order-three tensors}
	\author{Benjamin Biaggi}
	\address{Mathematical Institute, University of Bern, Alpeneggstrasse 22, 3012 Bern, Switzerland}
	\email{benjamin.biaggi@unibe.ch}
	
	\author{Jan Draisma}
	\address{Mathematical Institute, University of Bern, Sidlerstrasse 5, 3012 Bern, Switzerland}
	\email{jan.draisma@unibe.ch}
	
	\author{Sarah Eggleston}
	\address{Institute of Mathematics, Osnabrück University, Albrechtstra{\ss}e 28a, 49076 Osnabr\"uck, Germany}
	\email{sarah.eggleston@uni-osnabrueck.de}
	
	\thanks{BB and JD were supported by Swiss National Science foundation
		project grant 200021-227864. SE was partially funded by the Deutsche Forschungsgemeinschaft (DFG) – Project 445466444.}
	
	
	\begin{abstract}
		We study the subrank of real order-three tensors and give a lower bound to the subrank of a real tensor given its complex subrank. Using similar arguments to those used by Bernardi-Blekherman-Ottaviani, we show that all subranks between the minimal typical subrank and the maximal typical subrank, which equals the generic subrank, are also typical. We then study small tensor formats with more than one typical subrank. In particular, we construct a $3 \times 3 \times 5$-tensor with subrank $2$ and show that the subrank of the $4 \times 4 \times 4$-quaternion multiplication tensor is $2$. Finally, we consider the tensor associated to componentwise complex multiplication in $\CC^n$ and show that this tensor has real subrank $n$---informally, {\em no more than $n$ real scalar multiplications can be carried out using a device that does $n$ complex scalar multiplications}. We also prove a version of this result for other real division algebras.
	\end{abstract}
	
	\maketitle

	\section{Introduction}
	The rank of a matrix 
	can be defined in many different but equivalent ways, such as the minimum number of rank-one matrices required to sum to the given matrix, or the number of ones on the diagonal achieved through Gaussian elimination. For tensors of larger order, these definitions no longer coincide, and many different generalizations of matrix rank have been studied over the last few decades. 
	The subrank, first introduced by Strassen \cite{Strassen1987}, is
	defined as the length of the largest diagonal tensor that can be
	obtained from the tensor by linear slice operations in all directions.
	This is dual to the rank of a tensor, which is defined as the minimal
	number of rank-one tensors that sum to the tensor. 
	
	Given a $k$-vector space $V$, an order-three tensor $T \in V^\ast
	\otimes V^\ast  \otimes V$ defines a bilinear map $V \times V \to V$
	also denoted $T$. The subrank is the largest $r$ such that there exist linear maps $\varphi_1 : k^r \to V, \varphi_2: k^r \to V$ and $\varphi _3:V \to k^r$ with $\varphi _3 (T(\varphi_1 (a), \varphi_2 (b) )) = I_r(a,b)$ for all $a,b \in k^r$, where $I_r$ is the bilinear map given by $r$ independent scalar multiplications:
	\[
	I_r : k^r \times k^r \to k^r; ((a_1 , \ldots , a_r),(b_1 , \ldots , b_r) ) \mapsto (a_1 b_1, \ldots , a_r b_r).
	\]
	The rank of $T$ is dually the smallest $l$ such that there exist
	linear maps $\psi_1 :  V \to k^l$, $\psi_2 :  V \to k^l$ and $\psi_3 :
	k^l \to V$ with $ T(v,w) = \psi _3 (I_l(\psi_1 (v), \psi_2 (w) )) $
	for all $v,w \in V$. Thus the rank corresponds to the minimal number
	of linearly independent scalar multiplications required to evaluate $T$, or the ``cost'' of the tensor, and the subrank is the maximal number of linearly independent scalar multiplications that can be embedded into $T$, the ``value'' of the tensor.
	
	In this paper, we focus on the case where $k=\RR$. Hence, 
	we discuss different aspects of the following natural question:
	\begin{quote}
		{\em 
			Given a real bilinear map, how many linearly independent $\RR$ scalar multiplication can be embedded into the bilinear map? 
		}
	\end{quote}
	Recent work by Derksen-Makam-Zuiddam \cite{DMZ24} and
	Pielasa-{\v{S}}afránek-Shatsila \cite{PSS-exactvaluesgenericsubrank}
	solved this question for sufficiently general tensors over
	algebraically closed fields; in particular, they showed that the
	{\em generic} subrank of $T \in k^{n} \otimes k^{n} \otimes k^{n}$ is
	$\lfloor\sqrt{3n -2 } \rfloor$. However, over $\RR$, there can be
	multiple {\em typical} subranks for certain tensor formats. Similarly,
	while over an algebraically closed field there is a single generic
	rank for a given tensor format, over $\RR$ there can be multiple
	typical ranks. In addition to several general results, we will
	investigate many small formats of order-three tensors. We start by giving precise definitions for the terms introduced above.
	
	\begin{definition}
		Let $T\in V_1 \otimes \dots \otimes V_d $, where $V_i$ are vector spaces over the field $k$. Then the subrank of $T$ is defined as
		\[
		Q(T) := \max \{r \mid \exists \text{ linear maps } \varphi_i: V_i  \to k ^{r}  ; (\varphi_1 \otimes \dots \otimes \varphi_d)T = \sum_{j=1}^r e_j \otimes \dots \otimes e_j =: I_r \} .
		\]
		We call $I_r$ the \emph{$r$-th unit tensor}. 
	\end{definition}
	
	We recall that the rank of a tensor $T\in V_1 \otimes \dots \otimes V_d $ is the minimal integer $r$ for which $T$ is equal to the sum of $r$ rank-one tensors, which we denote by
	\[
	\rank (T) := \min \{r \mid \exists \, v_{ij} : T = \sum_{j = 1}^r v_{1j} \otimes \ldots \otimes v_{dj} \} .
	\]
	We will sometimes use the following description of the rank, which is
	analogous to the definition of the subrank. The rank of a tensor $T$ is the minimal integer $r$, for which there exist linear maps $\varphi _i :k^r \to V_i$ sending $I_r$ to $T$:
	\[
	\rank (T) = \min \{r \mid \exists \text{ linear maps } \varphi_i: k ^{r} \to V_i ; (\varphi_1 \otimes \dots \otimes \varphi_d) I_r = T \} .
	\]
	Using this description of the rank, one can think of the subrank as being a dual notion to the rank of a tensor. 
	
	\bigskip
	
	When $d=2$, the tensor rank and subrank equal the matrix rank of $T$.
	The definition implies the following two observations, which can be
	helpful to  bound the subrank from above. When regarding $T$ as a
	linear map $\bigotimes_ {j \neq i} V_j^\ast \to V_i$, the subrank
	$Q(T)$ is at most the rank of this map, and hence the subrank is
	bounded from above by $\min _i \dim (V_i)$. Second, assume that $T \in
	V_1 \otimes \ldots \otimes V_d$, all dimensions are equal ($n:= \dim
	(V_1) = \ldots = \dim (V_d)$), and the subrank is maximal ($Q(T) =
	n$), i.e.~there exist linear maps $\varphi _i$ with $(\varphi_1
	\otimes \dots \otimes \varphi_d)T = I_n$. Then it follows that $T$ is
	concise and that the linear maps are invertible, which implies that
	also the rank $R(T)$ of $T$ equals $n$. The same argument applies if $T$
	is concise and the rank of $T$ is minimal (namely, $n$), so we obtain:
	
	\begin{re}
		\label{rm:maxsubrank=rank}
		Let $T \in V_1 \otimes \ldots \otimes V_d$ be a tensor with $n = \dim (V_1) = \ldots = \dim (V_d)$; then 
		\[
		\rank (T) = n \text{ and $T$ is concise}\qquad  \Leftrightarrow \qquad Q(T) = n. \qedhere
		\]
	\end{re}

	In general, it is difficult to find upper bounds on the subrank of a given tensor. The \emph{slice rank} of $T \in V_1 \otimes \ldots \otimes V_d $ is the minimal sum $\sum _{i = 1 }^d \dim (U_i)$ where $U_i $ are linear subspaces of $V_i$ such that 
	\[
	T \in \sum _{i = 1 }^d V_1 \otimes \ldots \otimes V_{i-1} \otimes U_i
	\otimes V_{i+1} \otimes \ldots \otimes V_d; 
	\] 
	see \cite{Tao16}, where it is also shown that the slice rank $\leq r$
	locus is closed.  Using the fact that slice rank can only drop under
	linear transformations and that the slice rank of $I_r$ equals $r$,
	we find that the slice rank is an upper bound to the subrank. However,
	this bound is by no means tight for sufficiently general tensors: the
	generic subrank is strictly smaller than the generic slice rank: the
	latter equals $n$ when all $V_i$ have dimension $n$, while the former
	is $O(n^{1/(d-1)})$ \cite{DMZ24}.
	
	In the remainder of this paper we will restrict ourselves to the
	subrank of real tensors $T \in  \RR ^{n_1} \otimes \dots \otimes \RR
	^{n_d}$, so all the linear maps in the definition are $\RR$-linear: 
	\[
	Q(T) = \max \{r \mid \exists \, \RR\text{-linear maps } \varphi_i: \RR ^{n_i}  \to \RR ^{r}  \text{ with } (\varphi_1 \otimes \dots \otimes \varphi_d)T = I_r \} .
	\]
	One can ask what happens if we relax the definition of subrank by
	allowing the linear maps to be complex, which leads us to the
	following definition:
	\begin{definition}
		\label{def:complex_subrank}
		Let $T\in \RR ^{n_1} \otimes \dots \otimes \RR ^{n_d} $. Then the \emph{complex subrank} of a real tensor $T$ is defined as
		\[
		Q_\CC (T) := \max \{r \mid \exists \, \CC\text{-linear maps } \varphi_i: \CC ^{n_i}  \to \CC ^{r}  ; (\varphi_1 \otimes \dots \otimes \varphi_d)T = I_r \} .
		\]
	\end{definition}
	
	For real matrices $A \in \RR^{n_1} \otimes \RR^{n_2}$ we know that the subrank equals the matrix rank of $A$, and this is the same over all possible field extensions, so $Q(A) = Q_\CC (A)$. However, for higher-order tensors, the complex subrank only provides an upper bound on the real subrank.
	\begin{ex}\label{ex:subrankC}
		The real tensor 
		\[
		T:= (e_1 + i e_2)^{\otimes 3} + (e_1 - i e_2)^{\otimes 3} = 2 (e_1
		\otimes e_1 \otimes e_1 - e_1 \otimes e_2 \otimes e_2 - e_2
		\otimes e_1 \otimes e_2 - e_2 \otimes e_2 \otimes e_1)  \in \RR
		^2 \otimes \RR ^2 \otimes \RR ^2  \] has $Q(T)=1$ as $R(T)=3$, which can be shown using the Cayley hyperdeterminant (see Section~\ref{sec:2xnxm}).
		However, we have $R_\CC(T)=2$ and hence $Q_\CC(T)=2$ by 
		Remark~\ref{rm:maxsubrank=rank}, where $\R_\CC$ is defined analogue to $Q_\CC$. 
	\end{ex}
	
	Let $T \in \RR ^n \otimes \RR ^n \otimes \RR ^n$. For tensor rank, it
	is known that the rank of a real order-$d$ tensor is bounded by $d$
	times its complex rank \cite{Ballico-upperboundrealtensorrank}. So it
	is natural to expect a similar result for subrank. 
	In Section~\ref{sec:bound}, we establish such a result, as follows:
	
	\begin{thm}[Theorem~\ref{thm:bound} below]
		For any order-three tensor $T$ we have 
		\[Q(T) \geq \lfloor \sqrt{ Q_\CC (T)} \rfloor.\]
	\end{thm}
	
	In the first version of this paper, we expressed the belief that this bound should be far from sharp, and indeed, in view of Proposition~\ref{prop:maxcompq} below, that there might be a lower bound on $Q(T)$ that is {\em linear} in $Q_{\CC}(T)$. But in the meantime, a construction communicated to us by Jeroen Zuiddam has established that the square-root bound is sharp up to a multiplicative constant; see Example \ref{ex:sqrt}. 
	
	\bigskip 
	
	Over the complex numbers, the locus of tensors with subrank $r$ in
	$\CC ^{n_1} \otimes \CC ^{n_2} \otimes \CC ^{n_3}$ is a constructible
	set. Therefore, there exists a unique $r$ for which the locus of
	tensors with subrank $r$ is dense. This $r$ is called the generic
	subrank of tensors in $\CC ^{n_1} \otimes \CC ^{n_2} \otimes \CC
	^{n_3}$ \cite{DMZ24}. 
	
	Over the real numbers, there does not have to be a dense subset in
	the Euclidean topology where the subrank is constant. This leads us to the following definition.
	\begin{definition}
		A subrank $r$ is called \emph{typical} if the set of tensors in $\RR ^{n_1} \otimes \RR ^{n_2} \otimes \RR ^{n_3}$ with subrank $r$ contains a nonempty open subset with respect to the Euclidean topology.
	\end{definition}
	
	For tensor rank, we know that any rank between the smallest typical rank (which is the generic rank) and the largest typical rank is also typical \cite{BBO18}, where the definition of the typical rank is analogous to that of typical subrank. In Section~\ref{sec:consecutive}, we prove that the same is true for typical subranks:
	
	\begin{thm}
		\label{thm:consecutive}
		Let $r$ and $s$ be typical subranks of tensors in $\RR ^{n_1} \otimes \RR^{n_2} \otimes \RR^{n_3}$ with $r \leq s$. Then all integers $l $ with $r \leq l \leq s$ are also typical subranks.
	\end{thm}
	
	\bigskip

	Both 2 and 3 are typical ranks of real $2\times 2 \times 2$ tensors \cite{Kruskal89}; combining this fact with Remark~\ref{rm:maxsubrank=rank}, this gives us that 1 and 2 are both typical subranks of tensors in $\RR ^2 \otimes \RR ^2 \otimes \RR^2$. In Section~\ref{sec:subrankforspecificorder-threeformats}, we look at more examples of $n_1 \times n_2 \times n_3$ tensor spaces and describe the typical subranks. We summarize our findings in the following theorem.
	\begin{thm}
		Let $\RR ^{n_1} \otimes \RR ^{n_2} \otimes \RR ^{n_3}$ the space of real order-three tensors.
		\begin{enumerate}
			\item \label{it:222} For $(n_1,n_2,n_3) = (2,2,2)$, the typical subranks are 1 and 2.
			\item For $n_1 = 2$, $n_2 \geq 2$ and $ n_3 > 2$, the only typical subrank is 2.
			\item For $(n_1,n_2,n_3) \in \left\{ (3,3,3),(3,3,4)\right\}$, the only typical subrank is 2.
			\item For $(n_1,n_2,n_3) = (3,3,5)$, the typical subranks are 2 and 3.
			\item For $(n_1,n_2,n_3) \in \left\{ (3,4,4), (4,4,4)\right\}$, the typical subranks are 2 and 3.
		\end{enumerate}
	\end{thm}
	
	To give an upper bound for the subrank of some concrete tensors $T \in
	\RR ^{n_1} \otimes \RR ^{n_2} \otimes \RR^{n_3}$, we use a simple
	geometric condition for the subrank of $T$ to be at least some given
	integer $r$, namely, that the linear space spanned by the slices of the tensor along one axis contains at least $r$ linearly independent rank-one matrices.
	
	\bigskip
	In the final section of this paper, we turn to real division algebras
	and prove the following theorem.
	
	\begin{thm}[Theorem~\ref{thm:upperBoundQDivision} and
		Corollary~\ref{cor:subrankCandH} below]
		Let $D$ be a division algebra over $\RR$ of dimension $\geq 2$. Let $f_n$ be the componentwise multiplication map  
		\[
		f_n: D^n \times D^n \to D^n; (a,b)=((a_1, \ldots a_n),(b_1, \ldots , b_n)) \mapsto ( a_1  b_1, \ldots , a_n  b_n)=: a* b.
		\]
		Regarding $f_n$ as an $\RR$-bilinear map, we have $Q(f_n) \leq nd$,
		where $d = \frac{1}{2} \dim _\RR D$. In particular, for $D=\CC$ we have
		$Q(f_n) = n$ and for $D=\HH$, the quaternions, we have $Q(f_n)=2n$.
	\end{thm}
	
	We stress that this result for arbitrary $n$ does not trivially follow
	from the case where $n=1$, since the subrank of tensors can be strictly
	superadditive \cite{DMZ24}.

	\subsection*{Acknowledgements}
	
	We thank Jeroen Zuiddam for the construction in Example~\ref{ex:sqrt} below.

	\section{Real and complex subrank}
	
	\label{sec:bound}
	
	Given the complex subrank of a real tensor, we are interested in finding
	bounds on the real subrank. Let $T \in \RR^{n_1} \otimes \RR^{n_2}
	\otimes \RR^{n_3}$ be a tensor. As discussed above, we have the trivial
	upper bound $Q (T) \leq Q_\CC (T)$.
	
	In the following, we first consider an order-three tensor $T$ of any format $n_1 \times n_2 \times n_3$ and provide a lower bound of $Q(T)$ as a function of the complex subrank of $T$. We then show that this bound can be improved in the case that $Q_{\CC} (T)= n_1 = n_2 = n_3$. For these results, we use the following lemma.
	\begin{lem}
		\label{lem:spanningsetofslices}
		Let $T \in \RR^r \otimes \RR^s \otimes \RR ^s $ be a tensor where the  $r$ slices $T_1, \ldots , T_r \in \RR^s \otimes \RR ^s$ along the first axis are a spanning set of the space of $s\times s$ matrices. Then $Q(T) = s$.
	\end{lem}
	\begin{proof}
		For all $i \in [s]:=\{1,\ldots,s\}$, there exist coefficients $c_{i1},\dots,c_{ir} \in \RR$ such that $e_i \otimes e_i = E_{ii} = c_{i1} T_1 + \ldots + c_{ir} T_r$. Let $\varphi _1: \RR ^r \to \RR ^s$ be the linear map defined by the matrix
		\[
		\begin{pmatrix}
			c_{11} & \cdots & c_{1r} \\
			\vdots &  & \vdots \\
			c_{s1} & \cdots & c_{sr} \\
		\end{pmatrix} .
		\]
		Then 
		\[
		(\varphi _1 \otimes \id \otimes \id)T = (\varphi _1 \otimes \id \otimes \id)\sum_{i = 1}^r e_i \otimes T_i = \sum_{i=1}^s e_i \otimes (c_{i1} T_1 + \ldots + c_{ir} T_r) 
		= \sum_{i=1}^s e_i \otimes e_i \otimes e_i = I_s.
		\]
		This shows $Q(T) \geq s$, and equality holds, as the subrank is bounded from above by the dimension of $\RR ^s$.
	\end{proof}
	\begin{thm}
		\label{thm:bound}
		Let $T \in \RR^{n_1} \otimes \RR^{n_2} \otimes \RR^{n_3}$ be a real tensor. Then
		\[
		Q (T) \geq  \lfloor \sqrt{ Q_\CC (T)} \rfloor.
		\]
	\end{thm}
	
	\begin{proof}
		Set $r:=Q_\CC(T)$ and let $s$ be the largest integer such that $s^2
		\leq r$. We will show below that there exist {\em complex} linear maps
		$\phi_1 :\CC ^{n_1} \to \CC ^r$ and $\phi _{i}: \CC^{n_i} \to \CC^s$,
		for $i = 2,3$, such that the slices of $(\phi_1 \otimes \phi _2
		\otimes \phi _3)T$ along the first axis are a spanning set of the
		space of $s \times s$ matrices. Since this is a Zariski-open condition on the
		triple $(\phi_1,\phi_2,\phi_3)$ and the set of real matrices is
		Zariski dense in the set of complex matrices, this implies the
		existence of {\em real} linear maps $\rho_1 :\RR ^{n_1} \to \RR ^r$ and
		$\rho_{i}: \RR^{n_i} \to \RR^s\ (i=2,3)$ such that the slices of $(\rho_1 \otimes \rho_2 \otimes \rho_3)T$ span the set of $s\times s$ real matrices. 
		Lemma~\ref{lem:spanningsetofslices} then implies that $Q((\rho_1 \otimes \rho _2 \otimes \rho _3)T)\geq s$.  
		
		By definition of $r$, there exist complex linear maps $\tilde{\varphi} _i:\CC^{n_i} \to \CC^r$ such that $(\tilde{\varphi}_1 \otimes \tilde{\varphi} _2 \otimes \tilde{\varphi} _3 ) T = I_r$. There is an injective map $f:[s^2]   \to [s] \times [s] $ with $f(i)= (i,i)$ for $i \leq s$ and $f(i) = (j,k)$, where $j \neq k$, for $i > s$. 
		Let $P$ be the set of all triples $(i,j,k) = (i, f(i))$ with $i >
		s$ and $j,k \leq s$. Define for each triple $(i,j,k)$ in $P$ the map $\Psi_{ijk} :=
		\id \otimes \psi _{ij} \otimes \psi _{ik}$, where $\psi _{ij}  :
		\CC ^r \to \CC^r $ is the linear map defined by the matrix $I_r +
		E_{ji}$, and similarly for $\psi _{ik}$. Then 
		\[
		\Psi _{ijk} I_r = I_r + e_i \otimes e_j \otimes e_i + e_i \otimes e_i \otimes e_k + e_i \otimes e_j \otimes e_k.
		\]
		The following figure describes this image of $I_r$ under the map $\Psi _{ijk}$, where the elements 
		\[
		e_i \otimes e_i \otimes e_i, \qquad e_i \otimes e_j \otimes e_i,\qquad e_i \otimes e_i \otimes e_k,\qquad  e_i \otimes e_j \otimes e_k
		\]
		lie in the $i$-th horizontal slice shown in blue. 
		
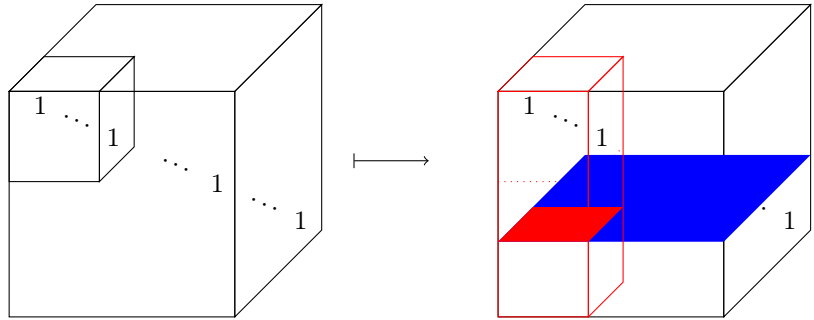
\[
		\begin{tikzpicture}
			\draw[black] (0,3,0) --(3,3,0) -- (3,3,3) --(0,3,3)--(0,3,0) -- cycle;
			\draw[black] (0,3,3) --(3,3,3) -- (3,0,3) --(0,0,3)--(0,3,3) -- cycle;
			\draw[black] (3,3,3) --(3,3,0) -- (3,0,0) --(3,0,3)--(3,3,3) -- cycle;
			\node  at (0.3, 2.7, 2.7) {1};
			\node  at (0.55, 2.45, 2.45) {$\cdot $};
			\node  at (0.65, 2.35, 2.35) {$\cdot $};
			\node  at (0.75, 2.25, 2.25) {$\cdot $};
			\node  at (1, 2, 2) {1};
			
			\draw[black] (1.2,3,3) --(1.2,3,1.8) -- (1.2,1.8,1.8) -- (1.2,1.8,3) -- cycle;
			\draw[black] (1.2,3,1.8) --(0,3,1.8) -- (0,3,3) -- (1.2,3,3) -- cycle;
			\draw[black] (0,3,3) --(1.2,3,3) -- (1.2,1.8,3) --(0,1.8,3)-- cycle;
			
			\node  at (1.4, 1.6, 1.6) {};
			\node  at (1.6, 1.4, 1.4) {$\cdot $};
			\node  at (1.7, 1.3, 1.3) {$\cdot $};
			\node  at (1.5, 1.5, 1.5) {$\cdot $};
			\node  at (2.0, 1, 1) {1};
			\node  at (2.35, 0.65, 0.65) {$\cdot $};
			\node  at (2.45, 0.55, 0.55) {$\cdot $};
			\node  at (2.55, 0.45, 0.45) {$\cdot $};
			\node  at (2.8, 0.2, 0.2) {1};
			
			\draw[|->] (4,1.5,1.5)--(5,1.5,1.5);
			
			\draw[black] (6.5+0,3,0) --(6.5+3,3,0) -- (6.5+3,3,3) --(6.5+0,3,3)--(6.5+0,3,0) -- cycle;
			\draw[black] (6.5+0,3,3) --(6.5+3,3,3) -- (6.5+3,0,3) --(6.5+0,0,3)--(6.5+0,3,3) -- cycle;
			\draw[black] (6.5+3,3,3) --(6.5+3,3,0) -- (6.5+3,0,0) --(6.5+3,0,3)--(6.5+3,3,3) -- cycle;
			\node  at (6.5+ 0.3, 2.7, 2.7) {1};
			\node  at (6.5+0.55, 2.45, 2.45) {$\cdot $};
			\node  at (6.5+0.65, 2.35, 2.35) {$\cdot $};
			\node  at (6.5+0.75, 2.25, 2.25) {$\cdot $};
			\node  at (6.5+1, 2, 2) {1};
			
			\draw[red, dotted] (6.5,1.8,3) --(6.5+1.2,1.8,3) -- (6.5+1.2,1.8,1.8) ;
			\draw[red] (6.5+1.2,3,1.8) --(6.5+0,3,1.8) -- (6.5+0,3,3) -- (6.5+1.2,3,3) -- cycle;

			\node  at (6.5+1.4, 1.6, 1.6) {};
			\node  at (6.5+1.6, 1.4, 1.4) {$\cdot $};
			\node  at (6.5+1.7, 1.3, 1.3) {$\cdot $};
			\node  at (6.5+1.5, 1.5, 1.5) {$\cdot $};
			\node[blue]  at (6.5+2, 1, 1) {1};
			\node  at (6.5+2.35, 0.65, 0.65) {$\cdot $};
			\node  at (6.5+2.45, 0.55, 0.55) {$\cdot $};
			\node  at (6.5+2.55, 0.45, 0.45) {$\cdot $};
			\node  at (6.5+2.8, 0.2, 0.2) {1};
			
			\node[blue]  at (6.5+2, 1, 2.5) {1};
			\node[blue]  at (6.5+0.5, 1, 1) {1};
			\node[red]  at (6.5+0.5, 1, 2.5) {1};
			
			\fill [blue,fill opacity=.1] (6.5, 1, 0) -- (6.5, 1, 3) -- (6.5+3, 1, 3) -- (6.5+3, 1, 0) -- cycle;
			
			\fill [red,fill opacity=.1] (6.5, 1, 3) -- (6.5, 1, 1.8) -- (6.5+1.2, 1, 1.8) -- (6.5+1.2, 1, 3) -- cycle;
			
			\draw[red] (6.5+1.2,0,3) --(6.5+1.2,3,3) -- (6.5+1.2,3,1.8) -- (6.5+1.2,0,1.8) -- cycle;
			\draw[red] (6.5+0,0,3) --(6.5+1.2,0,3) -- (6.5+1.2,3,3) --(6.5+0,3,3)-- cycle;

		\end{tikzpicture}
		\]
		
		Define $\Psi:= \prod _{(i,j,k) \in P} \Psi _{ijk}$; then the projection of $\Psi (I_r) =\Psi \circ (\tilde{\varphi_1} \otimes \tilde{\varphi}_2 \otimes \tilde{\varphi}_3) T$ to $\CC ^r \otimes \CC^s \otimes \CC ^s$ is the tensor
		\[
		I_s + \sum _{(i,j,k) \in P} e_i \otimes E_{jk} = \sum _{l = 1}^s e_l \otimes E_{ll} +\sum _{(i,j,k) \in P} e_i \otimes E_{jk}
		.
		\]
		The set $\{E_{jk} \in \CC^s \otimes \CC^s: 1 \leq j,k \leq s\}$ is
		a spanning set as desired (these are the horizontal slices in the red rectangular cuboid seen in the picture above). 
		This shows that $Q (T) \geq s = \lfloor \sqrt{r} \rfloor = \lfloor \sqrt{Q_\CC (T)} \rfloor$.   
	\end{proof}
	
	The following construction was communicated to us by Jeroen Zuiddam, who found it using ChatGPT~Pro. 
	
	\begin{ex} \label{ex:sqrt}
		Let $m$ be a nonnegative integer. We construct a tensor in $\RR^{2m}
		\otimes \RR^{2m} \otimes \RR^m$ as follows. Let $U$ be an $m$-dimensional
		$\CC$-subspace of $\CC^{2m}$. Then also $\overline{U}:=\{\overline{u}
		\mid u \in U\}$ is an $m$-dimensional $\CC$-subspace of $\CC^{2m}$.
		We choose $U$ in such a manner that $\CC^{2m}=U \oplus \overline{U}$; for
		instance, we could take $U$ spanned by the vectors $e_j+ie_{m+j}$ for
		$j=1,\ldots,m$. Let $u_1,\ldots,u_m$ be a $\CC$-basis of $U$ and define
		\[ T_0:=\sum_{i=1}^m u_i \otimes u_i \otimes e_i \in U \otimes_\CC U
		\otimes_\CC \CC^m \subseteq \CC^{2m} \otimes_\CC \CC^{2m} \otimes_\CC \CC^m, \]
		a tensor isomorphic to the $m$-th complex unit tensor.  Now let $S$ be
		a tensor in $U \otimes_\CC \overline{U} \otimes_\CC \CC^m$ to be chosen
		later, and set
		\[ T(S):=T_0 + \overline{T_0} + S + \overline{S}
		\in (U \otimes_\CC U + 
		\overline{U} \otimes_\CC \overline{U} + 
		U \otimes_\CC \overline{U} + 
		\overline{U} \otimes_\CC U ) \otimes_\CC \CC^m. \]
		This is {\em a priori} just a tensor in $\CC^{2m} \otimes_\CC \CC^{2m}
		\otimes_\CC \CC^m$, but since it is preserved under complex conjugation,
		it actually lives in $\RR^{2m} \otimes_\RR \RR^{2m} \otimes_\RR \otimes
		\RR^m$.
		
		We first show that $Q_\CC(T(S)) \geq m$. Indeed, let $\phi:\CC^{2m} \to U$
		be the $\CC$-linear projection with kernel $\overline{U}$. Then it
		follows that
		\[ (\phi \otimes \phi \otimes \id_{\CC^m}) T(S) 
		= T_0 + 0 + 0 + 0 = T_0 \]
		by virture of the fact that $\phi$ is the identity on $U$ and zero on
		$\overline{U}$.
		
		Next we show that for $S$ sufficiently general, $Q_\RR(T(S))$ is bounded
		from above by $\sqrt{20m}$. This is a parameter count
		reminiscent of that used in \cite{DMZ24} to upper bound the generic subrank of a complex tensor. Fix an integer $r \leq m$ and surjective
		maps $\phi_1,\phi_2:\RR^{2m} \to \RR^r$ and $\phi_3:\RR^m \to \RR^r$,
		and set $\Phi:=\phi_1 \otimes \phi_2 \otimes \phi_3$. The set of all
		$S$ with $\Phi(T(S))=I_r$ is a real affine space. We
		determine an upper bound on its dimension.
		
		Let $\psi_1:U \to \CC^r,\ \psi_2:\overline{U} \to
		\CC^r$ be the restrictions to $U,\overline{U}$ of the complexifications
		of $\phi_1,\phi_2$, respectively, and let $\psi_3:\CC^m \to \CC^r$ be the complexification
		of $\phi_3$. Every vector $x \in \CC^{2m}$ decomposes as $x=u + \overline{u'}$
		for unique $u,u' \in U$, and if $x$ is real, then it follows that
		$u=u'$. Since $\phi_1$ is surjective, we conclude that every vector
		in $\RR^r$ is of the form $\psi_1(u) + \overline{\psi_1(u)}$. This implies that the images of the $\CC$-linear map $\psi_1$ and the semilinear map $u
		\mapsto \overline{\psi_1(u)}$ together span $\CC^{r}$. Since these spaces have the
		same dimension, the rank of $\psi_1$ is at least $\frac{r}{2}$.  The same
		applies to $\psi_2$, while $\psi_3$ has rank $r$. Now define
		\[ 
		\Psi:=\psi_1 \otimes \psi_2 \otimes \psi_3: U \otimes_\CC \overline{U}
		\otimes_\CC \CC^m \to \CC^r \otimes \CC^r \otimes \CC^r =: \CC^N.  \]
		By the above, $\Psi$ has rank at least $\frac{r}{2}\cdot \frac{r}{2} \cdot
		r=\frac{r^3}{4}$. We further have
		\[ \Phi(T(S))=\Phi(T_0+\overline{T_0}) + (\Psi(S) + \overline{\Psi(S)}), 
		\]
		and the second term is the real part of $\Psi(S)$. A basic fact from
		linear algebra is that the real parts of vectors in a complex subspace
		$W \subseteq \CC^N$ form a real subspace in $\RR^N$ of real dimension at
		least $\dim_\CC W$. Applying this with $W$ equal to the image of
		$\Psi$, we find that the second term exhausts a real space of real
		dimension at least $r^3/4$. So, since $S$ ranges through a space of
		real dimension $2m^3$, the $S$ with $\Phi(T(S))=I_r$ form a real affine space of dimension at most $2m^3-r^3/4$.
		
		Finally,
		$(\phi_1,\phi_2,\phi_3)$ varies through (an open subset of) a
		real space of dimension $2mr+2mr+mr=5mr$. Hence, the $S$ for which
		$Q_\RR(T(S)) \geq r$ form a semi-algebraic set of dimension at most
		$2m^3-r^3/4+5mr=2m^3-r(r^2-20m)/4$. Hence for $r>\sqrt{20m}$ there exist
		$S$ such that $T(S)$ has real subrank $<r$, and indeed this then holds for any sufficiently general $S$.
	\end{ex}
	
	For $n\times n \times n$ tensors with maximal complex subrank $n$, there exists a better linear bound:

	\begin{prop}
		\label{prop:maxcompq}
		Let $T \in \RR^n \otimes \RR^n \otimes \RR^n$ be a real tensor with maximal complex subrank $n$; then $Q (T) \geq \lceil n/2 \rceil.$
	\end{prop}
	
	\begin{proof}
		Let $T \in \RR^n \otimes \RR^n \otimes \RR^n$ be a tensor with
		maximal complex subrank $n$. Then by Remark~\ref{rm:maxsubrank=rank}
		the tensor has complex rank $n$
		and can be written as $T = \sum _{i=1}^n u_i \otimes v_i \otimes
		w_i$, where the vectors $u_i, v_i, w_i$ may be complex and are
		linearly independent. As $T$ is a real tensor, we have $T =
		\overline{T}$ and hence  
		\[
		\sum _{i=1}^n u_i \otimes v_i \otimes w_i = \sum _{i=1}^n \overline{u_i} \otimes \overline{v_i} \otimes \overline{w_i}
		\]
		Using Hitchcock's theorem \cite{Hitchcock27}, the decomposition of $T$
		as a sum of $n$ rank-one tensors is unique up to rearrangement, so for all $i$, we have $u_i \otimes v_i \otimes w_i = \overline{u_j} \otimes \overline{v_j} \otimes \overline{w_j}$ for some $j$. Using this, after a possible reordering, we have
		\[
		T = \sum_{i = 1}^{m} u_i \otimes v_i \otimes w_i + \sum_{j = m+1}^{m + l} (u_j \otimes v_j \otimes w_j + \overline{u_j} \otimes \overline{v_j} \otimes \overline{w_j}),
		\]
		where $n = m +2l$ and $u_i, v_i, w_i \in \RR ^n $ for all $i \leq
		m$. Now apply the real linear map defined by
		\begin{align*}
			\varphi _1 (u_j ) &= e_j, &j \leq m\\
			\varphi _1 (u_{m+j} + \overline{u_{m+j} } ) &= e_{m+2j-1}, &i \leq j \leq l\\
			\varphi _1 \left(\frac{1}{i} (u_{m+j} - \overline{u_{m+j} }) \right) &= e_{m+2j}, &i \leq j \leq l
		\end{align*}
		and similarly for $\varphi _2$ and $\varphi _3$, we can map the tensor $T$ to a tensor with $m$ ones on the diagonal, $l$ blocks of the size $2\times 2 \times 2$ and 0 everywhere else. 
		
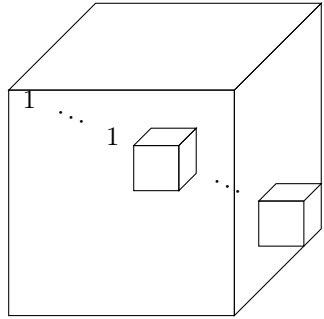
\[
		\begin{tikzpicture}
			\draw[black] (0,3,0) --(3,3,0) -- (3,3,3) --(0,3,3)--(0,3,0) -- cycle;
			\draw[black] (0,3,3) --(3,3,3) -- (3,0,3) --(0,0,3)--(0,3,3) -- cycle;
			\draw[black] (3,3,3) --(3,3,0) -- (3,0,0) --(3,0,3)--(3,3,3) -- cycle;
			\node  at (0.2, 2.8, 2.8) {1};
			\node  at (0.5, 2.5, 2.5) {$\cdot $};
			\node  at (0.6, 2.4, 2.4) {$\cdot $};
			\node  at (0.7, 2.3, 2.3) {$\cdot $};
			\node  at (1, 2, 2) {1};
			
			\draw[black] (1.2,1.8,1.2) --(1.2,1.8,1.8) -- (1.8,1.8,1.8) -- (1.8,1.8,1.2) -- cycle;
			\draw[black] (1.8,1.8,1.8) --(1.8,1.2,1.8) -- (1.2,1.2,1.8) -- (1.2,1.8,1.8) -- cycle;
			\draw[black] (1.8,1.8,1.8) --(1.8,1.8,1.2) -- (1.8,1.2,1.2) --(1.8,1.2,1.8)-- cycle;
			
			\node  at (2.0, 1.0, 1.0) {$\cdot $};
			\node  at (2.1, 0.9, 0.9) {$\cdot $};
			\node  at (2.2, 0.8, 0.8) {$\cdot $};
			
			\draw[black] (2.4,0.6,0.6) --(2.4,0.6,0) -- (3,0.6,0) -- (3,0.6,0.6) -- cycle;
			\draw[black] (3,0,0.6) --(2.4,0,0.6) -- (2.4,0.6,0.6) --(3,0.6,0.6) -- cycle;

		\end{tikzpicture}
		\]
		This shows that the real subrank is $\geq m + l$, so at least $\geq \lceil n/2 \rceil$.
	\end{proof}
	
	\begin{re}
		The bound in Proposition~\ref{prop:maxcompq} is tight. To see
		this, consider the $2n \times 2n \times 2n$ real tensor associated
		to the real bilinear map $f: \CC^{n} \times \CC^{n} \to \CC^{n}$
		defined by componentwise complex multiplication. This tensor has
		(real) subrank $n$ (see Theorem~\ref{cor:subrankCandH} below) and
		complex subrank $2n$; the latter statement follows using
		Remark~\ref{rm:maxsubrank=rank} and from the fact that the complex
		rank of $\CC^{1} \times \CC^{1} \to \CC^{1}$ is $2$, since $\CC
		\otimes_\RR  \CC$ is isomorphic to $\CC \times \CC$ as a
		$\CC$-algebra.
	\end{re}
	
	\section{Typical subranks are consecutive}
	\label{sec:consecutive}

	In \cite{BBO18}, the authors showed that any rank between the minimal
	and maximal typical rank is also typical. Their argument uses that
	sets of tensor with a given typical rank are semialgebraic and given a
	tensor, adding a rank-one tensor increases the rank by at most one. To
	use their proof idea for subrank, we dually need that the subrank does
	not decrease by more than one when we add a rank-one tensor.
	
	The following argument works for all order $d\geq 2$ tensors over any field, but for consistency, we only state it for real $n_1 \times n_2 \times n_3 $ tensors.
	\begin{lem}
		\label{lm:addingsimpletensor}
		Let $T\in \RR^{n_1} \otimes \RR^ {n_2} \otimes \RR ^{n_3}$ be a tensor and $v$ a rank-one tensor, then $Q(T+v) \geq Q(T)-1.$
	\end{lem}
	\begin{proof}
		Let $T$ with $Q(T) \geq r$. Then there exists a linear map
		$\varphi := \varphi _1 \otimes \varphi _2 \otimes \varphi _3$,
		where $\varphi _i : \RR ^{n_i} \to \RR ^r$, s.t. $\varphi (T) =
		I_r$. Let $v = v_1 \otimes v_2 \otimes v_3$ be a rank-one tensor; we have 
		\[
		\varphi (T + v) = \varphi (T) + \varphi (v) = I_r + \varphi _1 (v_1) \otimes \varphi_2 (v_2) \otimes \varphi _3(v_3).
		\]
		If $\varphi (v)=0$, then we are done, so assume that $\varphi _i
		(v_i) \neq 0$ for all $i=1,2,3$.  Define the projection $\pi _1 :
		\R ^r \rightarrow \R ^r / \langle \varphi _1(v_1) \rangle$. Then
		\[
		(\pi _1 \otimes \id \otimes \id) (I_r + \varphi _1 (v_1) \otimes \varphi _2(v_2) \otimes \varphi _3(v_3)) = \sum _{i=1}^r \pi _1(e_i) \otimes e_i \otimes e_i.
		\]
		The vectors $\pi _1(e_1) , \ldots , \pi _1 (e_r) $ span $\R ^r / \langle \varphi _1(v_1) \rangle$ and are linearly dependent, so there exists $j$ such that $\pi _1(e_1) , \ldots , \pi _1(e_{j-1}) ,  \pi _1 (e_{j+1}), \ldots,\pi _1 (e_r)$ is linearly independent. Define the projection $ \pi _2 : \R^r \rightarrow \R^r /\langle e_j \rangle$ and define
		\[
		(\id \otimes \pi _2 \otimes \id) \left( \sum _{i=1}^r \pi _1(e_i) \otimes e_i \otimes e_i \right) =  \sum _{i=1, i \neq j}^r \pi _1(e_i) \otimes (e_i +\langle e_j \rangle)  \otimes e_i,
		\]
		which is a tensor with linearly independent vectors, and thus can be linearly mapped to $I_{r-1}.$ This shows that $Q(T + v) \geq r- 1.$
	\end{proof}
	
	Next, we want to show that the set of tensors with subrank at least $r$, denoted by $\mathcal{C}_r:= \{T \in \RR ^{n_1} \otimes \RR^{n_2} \otimes \RR^{n_3} \mid Q(T) \geq r\}$, is semialgebraic and that the generic subrank is the largest typical one. We recall the following construction from \cite{DMZ24}. On the set of tensors with nonzero entries on the main diagonal up to position $(r,r,r)$ and zeros everywhere else
	\[
	X_r = \{ T \in \RR ^{n_1} \otimes \RR^{n_2} \otimes \RR^{n_3}  \mid T_{ijk} = 0 \text{ for } (i,j,k) \in [r] \setminus \{ (i,i,i) \mid i \in [r]\} \text{ and } T_{iii} \neq 0 \text{ for } i \in [r] \},
	\]
	we define the linear transformation map
	\begin{align*}
		\psi _r : \GL _{n_1} \times \GL _{n_2} \times \GL_{n_3} \times X_r & \to \RR ^{n_1} \otimes \RR^{n_2} \otimes \RR^{n_3}\\
		(A,B,C,T) & \mapsto (A\otimes B \otimes C ) T.
	\end{align*}
	
	\begin{lem}[{\cite[Lemma~2.2]{DMZ24}}]
		The image of $\psi _r$ equals $\mathcal{C}_r$.
	\end{lem}
	Using this, we can prove the following two facts.
	\begin{cor}
		The set of tensors with subrank $r$, $\{T \in \RR ^{n_1} \otimes \RR^{n_2} \otimes \RR^{n_3} \mid Q(T) = r\}$, is semi-algebraic.
	\end{cor}
	\begin{lem}
		The generic subrank in $\CC ^{n_1} \otimes \CC^{n_2} \otimes \CC^{n_3}  $ is the largest typical subrank in $\RR ^{n_1} \otimes \RR^{n_2} \otimes \RR^{n_3}  $.
	\end{lem}
	\begin{proof}
		The set of tensors with subrank strictly greater than the generic subrank are contained in an algebraic variety of dimension at most $n_1 n_2 n_3-1$ and so cannot contain a real open subset with dimension $n_1n_2n_3$. Therefore typical subranks cannot exceed the generic subrank.
		
		Let $r$ be the generic subrank; then the map $ \psi _r (\CC) : \GL _{n_1} (\CC) \times \GL _{n_2} (\CC) \times \GL_{n_3} (\CC) \times X_r (\CC) \to \CC^{n_1} \otimes \CC^{n_2} \otimes \CC^{n_3}$ is dominant. Here, we write 
		\[
		X_r (\CC) = \left\{ T \in \CC ^{n_1} \otimes \CC^{n_2} \otimes \CC^{n_3}  \mid T_{ijk} = 0 \text{ for } (i,j,k) \in [r] \setminus \{ (i,i,i) \mid i \in [r]\} \text{ and } T_{iii} \neq 0 \text{ for } i \in [r] \right\}
		\]
		for the complex extension of $X_r$. Hence there is a dense open subset on which the derivative of the map is surjective. The real points $\GL _{n_1} \times \GL _{n_2} \times \GL_{n_3} \times X_r $ are Zariski dense in $ \GL _{n_1} (\CC) \times \GL _{n_2} (\CC) \times \GL_{n_3} (\CC) \times X_r (\CC)$, so there must exist a real point where the differential is surjective. Using the implicit function theorem, this shows that there exists a nonempty open set in the Euclidean topology of $\RR ^{n_1} \otimes \RR^{n_2} \otimes \RR^{n_3}  $ where the subrank equals $r$.
		
		The set of tensors with subrank strictly greater than the generic subrank are contained in an algebraic variety of dimension at most $n_1 n_2 n_3-1$ and so cannot contain a real open subset with dimension $n_1n_2n_3$. Therefore typical subranks cannot exceed the generic subrank.
		
	\end{proof}
	
	We are now ready to prove Theorem~\ref{thm:consecutive}, which shows that all integers between the generic subrank and the minimal typical subrank are also typical. The proof uses the same idea as in \cite{BBO18}. We also use the following lemma from their work. Here, $\inter S$ refers to the interior of $S$ and $\clos S$ the closure of $S$ in the Euclidean topology.
	\begin{lem}[{\cite[Lemma~2.1]{BBO18}}]
		\label{lm:intcl}
		Let $S$  be a semialgebraic set in $\RR ^k$. Then $\inter ( \clos S) \subseteq \clos (\inter S)$.
	\end{lem}
	\begin{proof}[Proof of Theorem~\ref{thm:consecutive}]
		Assume that $r$ is a typical subrank, but $r-1 $ is not. Therefore there is no nonempty open subset of tensors with subrank $r-1$. As above, let $\mathcal{C}_r = \{T \in \RR^{n_1} \otimes \RR^{n_2} \otimes \RR^{n_3} \mid Q(T) \geq r\}$ be the set of tensors with subrank at least $ r$. The fact that $r-1$ is not a typical subrank implies that $\inter (\mathcal{C}_{r-1} ) \setminus \mathcal{C}_r$ 
		has empty interior, hence $\inter \mathcal{C}_{r-1} \subseteq \clos \mathcal{C}_{r}$ and so $ \inter \mathcal{C}_{r-1} \subseteq \inter (\clos \mathcal{C}_{r})$. By Lemma~\ref{lm:intcl} $\inter (\clos \mathcal{C}_{r}) \subseteq \clos (\inter \mathcal{C}_r)$ and so
		\begin{equation}
			\label{eq:inlucsionC_r-1}
			\inter \mathcal{C}_{r-1} \subseteq \inter (\clos \mathcal{C}_{r}) \subseteq \clos (\inter \mathcal{C}_r) .
		\end{equation}
		Given a tensor $T \in \clos (\inter \mathcal{C}_r)$, there exists
		a sequence of tensors $\{T_i\} \subset \inter \mathcal{C}_r$
		converging to $T$ in the Euclidean topology. Let $v$ be a rank-one
		tensor; then by Lemma~\ref{lm:addingsimpletensor} $Q(T_i + v) \geq
		r-1$, and as $T_i \in \inter \mathcal{C}_{r}$, it follows that
		$T_i + v \in \inter \mathcal{C}_{r-1}$. Using the inclusion
		(\ref{eq:inlucsionC_r-1}), we have $T +v \in \clos (\inter
		\mathcal{C}_r)$. Inductively, we have for all $k \in \NN$ and for
		all rank-one tensors $v_i$
		\[
		T + \sum _{i=1}^k v_i \in \clos (\inter \mathcal{C}_r).
		\]
		This shows that all tensors are contained in the set $ \clos (\inter \mathcal{C}_r)$, so $ \clos (\inter \mathcal{C}_r) = \RR^{n_1} \otimes \RR^{n_2} \otimes \RR^{n_3} . $
		
		This implies that $r$ is the smallest typical subrank, as there cannot exist a nonempty open subset with tensors of subrank $<r$ if $ \clos (\inter \mathcal{C}_r) = \RR^{n_1} \otimes \RR^{n_2} \otimes \RR^{n_3}$; hence, all typical subranks are consecutive.
	\end{proof}

	\section{Subrank for specific order-three formats}
	\label{sec:subrankforspecificorder-threeformats}
	
	In this section, we consider different examples of spaces of real $n_1 \times n_2 \times n_3$ tensors and determine their typical subranks. We recall that the generic subrank of tensors in $\CC ^{n_1} \otimes \CC ^{n_2} \otimes \CC ^{n_3} $ is the largest typical subrank. As recently proved by \cite{PSS-exactvaluesgenericsubrank}, the generic subrank is $\min \{\lfloor \sqrt{n_1 + n_2 + n_3 -2 }\rfloor, n_1, n_2, n_3 \}$.
	
	The following lemma, which shows that the typical subranks behave well when scaling $n_1, n_2, n_3$, will be key to the subsequent arguments:
	
	\begin{lem}
		\label{lem:scaling}
		Let $m_1 \leq n_1, m_2 \leq n_2, m_3 \leq n_3$ be natural numbers. Then the minimal typical subrank of $n_1 \times n_2 \times n_3$ tensors is at least the minimal typical subrank of $m_1 \times m_2 \times m_3$ tensors.
	\end{lem}
	\begin{proof}
		Let $r$ be the minimal typical subrank in $\RR^{n_1}\otimes \RR^{n_2}\times \RR^{n_3}$. Then there exists a full-dimensional open subset $U \subset \RR ^{n_1} \otimes \RR ^{n_2} \otimes \RR ^{n_3} $ such that for all $T \in U$, $Q(T) = r$. Let $\pi_i : \RR^{n_i} \to \RR^{m_i}$ be the coordinate projection onto the first $m_i$ coordinates for $i=1,2,3$. The map $\pi_1 \otimes \pi_2 \otimes \pi_3$ is a surjective linear map, so $(\pi_1 \otimes \pi_2 \otimes \pi_3)U \subseteq \RR ^{m_1} \otimes \RR ^{m_2} \otimes \RR ^{m_3} $ is also open and of full dimension, and $Q(S) \leq r$ for all $S \in \pi(U)$. Hence, there exists an $s \leq r$ and a full-dimensional open subset $V \subseteq \pi(U)$ such that $Q(S) = s$ for all $S \in V$. It follows that the minimal typical subrank of $m_1 \times m_2 \times m_3$ tensors is at most $r$.
	\end{proof}
	
	Combining the above lemma with the upper bound of $\min \{ \lfloor \sqrt{n_1 + n_2 + n_3 -2 }\rfloor,  n_1, n_2 , n_3 \}$, we can deduce the typical subranks of some tensors formats when we know the typical subranks of smaller formats.

	\subsection{Typical subranks of \boldmath{$2\times m\times n$} tensors}
	\label{sec:2xnxm}
	
	We begin this section with the smallest nontrivial format of order-three tensors. The typical ranks of $2 \times 2 \times 2$ tensors are 2 or 3, corresponding to $\Delta(T) > 0$ and $\Delta(T) < 0$, respectively, where $\Delta(T)$ is the Cayley hyperdeterminant \cite[Propositions 5.9, 5.10]{Silva-Lim08}. Furthermore, generically there exists a one-to-one correspondence between rank and subrank; for generic tensors, we have that $\rank (T) = 2 $ precisely when $Q(T) = 2$ (see Remark~\ref{rm:maxsubrank=rank}), which implies that $\rank (T) = 3$ precisely when $Q(T) = 1$. 
	\begin{prop}
		For $2\times 2 \times 2$ tensors, the typical subranks are $1$ and $2$. For a tensor $T = (t_{ijk})_{ijk}$ with $t_{ijk}$ i.i.d.\ Gaussian random variables, $\mathbf{P}(Q(T) = 2) = \frac{\pi}{4}$. 
	\end{prop}
	
	\begin{proof}
		For $T \in \mathbb{R}^2 \otimes \mathbb{R}^2 \otimes \mathbb{R}^2$ outside some Zariski-closed proper subset, $T$ is concise. Now 2 is the maximal subrank, so
		\[
		Q(T) = 2 \;\;\; \Leftrightarrow \;\;\; \rank (T) = 2
		\]
		by Remark~\ref{rm:maxsubrank=rank}. Thus, as calculated by \cite{Bergqvist2013},
		\[
		\mathbf{P}(Q(T) = 2) = \mathbf{P}(\rank (T) = 2) = \frac{\pi}{4}. \qedhere
		\]
	\end{proof}
	
	\begin{lem}\label{lm:223}
		The unique typical subrank for $2 \times 2 \times 3 $ tensors is 2.
	\end{lem}
	\begin{proof}
		Typical $2\times 2\times 3$-tensors have rank 3 \cite{tenBerge-Kiers99}. Fix vector spaces
		$V_1,V_2,V_3$ of dimensions $2,2$ and $3$, respectively. Let $T\in
		V_1 \otimes V_2\otimes V_3$ be a sufficiently general tensor; then we can write a tensor rank decomposition of $T$ as
		\[
		T = \sum_{i=1}^3 u_i \otimes v_i \otimes w_i
		\]
		where $u_i \in V_1$, $v_i\in V_2$ and $w_i \in V_3$. We may
		further assume that $u_1,u_2$ are linearly independent, $v_1,v_2$ are linearly independent, and $w_1,w_2,w_3$ are all linearly independent.
		
		We define the following linear functions $\varphi_i : V_i \to \R^2$ for $i=1,2,3$:
		\begin{align*}
			&\varphi_1(u_1) := e_1, \;\;\;\; \varphi_1(u_2) := e_2\\
			&\varphi_2(v_1) := e_1, \;\;\;\; \varphi_2(v_2) := e_2\\
			&\varphi_3(w_1) := e_1, \;\;\;\; \varphi_3(w_2) := e_2, \;\;\;\; \varphi_3(w_3) := 0,
		\end{align*}
		which shows that $Q(T) \geq 2$.
		The subrank is bounded from above by $\min_i \{\text{dim }V_i\} = 2$, so it follows that $Q(T) = 2$.  Because this is the case for all sufficiently general $2\times 2\times 3$ tensors, this is the only typical subrank.
	\end{proof}
	
	\begin{thm}
		\label{thm:2xmxn}
		The unique typical subrank for $2\times m\times n$ tensors with $m,n \geq 2$ and $(m,n) \neq (2,2)$ is 2.
	\end{thm}
	
	\begin{proof}
		Without loss of generality, assume that $2 \leq m$, $3\leq n$ and
		let $T \in \RR^2 \otimes\RR^m \otimes\RR^n$ be typical. Then $Q(T)$ is bounded from above by $2 = \text{dim}(V_1)$ and using Lemma~\ref{lem:scaling} together with the fact that the minimal typical subrank of $2 \times 2 \times 3$ tensors is 2, it follows that $Q(T) = 2$.
	\end{proof}

	From this result, we can also conclude the typical subranks of $3 \times 3 \times 3$ and $3 \times 3 \times 4$ tensors. 
	\begin{cor}
		The unique typical subrank for $3 \times 3 \times 3$ and $3 \times 3 \times 4$ tensors is 2.    
	\end{cor}

	\begin{proof}
		
		By the upper bound of \cite{PSS-exactvaluesgenericsubrank}, the
		typical subranks are in both cases at most $\lfloor \sqrt{n_1 + n_2 + n_3-2} \rfloor= 2$.
		Using Theorem~\ref{thm:2xmxn}, the typical subrank of tensors in $\RR ^2 \otimes \RR ^3 \otimes \RR^3$ is 2 and hence by Lemma~\ref{lem:scaling}, the smallest typical subranks of $3 \times 3 \times 3$ tensors is at least 2.
	\end{proof}
	
	\begin{re}
		Ten Berge and Kiers showed that the typical ranks $\RR ^2 \otimes \RR ^m \otimes \RR ^m$ are $\{m, m+1\}$ \cite{tenBerge-Kiers99}. Hence a format having multiple typical ranks does not imply the existence of multiple typical subranks.
	\end{re}

	\subsection{A geometric method to bound the subrank from above}
	\label{sec:geometricmethod}
	There is no established method to give an upper bound for the subrank of a given tensor; hence, showing that a given integer that is strictly smaller than the generic subrank is a typical rank is difficult. In the following, we give a necessary condition for a tensor to have subrank at least $r$, which we use to give upper bounds on the subrank of $3 \times 3 \times 5$ and $4 \times 4 \times 4$ tensors. We prove that in those cases, there exist multiple typical subranks.
	
	\bigskip
	
	Let $T \in \RR ^{n_1} \otimes \RR ^{n_2} \otimes \RR ^{n_3} $ be a
	tensor and denote by $T_1, \ldots , T_{n_1} \in \RR ^{n_2} \otimes \RR
	^{n_3}$ the slices of $T$ along the first axis. Assume that
	the subrank of $T$ is $\geq r$. Then there exist surjective linear
	maps $\pi_2:\RR^{n_2} \to \RR^r$ and $\pi_3:\RR^{n_3} \to \RR^r$
	and an $r$-dimensional linear
	subspace $W \subset \langle \tilde{T_1}, \ldots, \tilde{T}_{n_1}
	\rangle$, where $\tilde{T_i}:=(\pi_2 \otimes \pi_3)(T_i)$, 
	which up to left-right $GL_r \times
	GL_r$ equivalence is of the form  
	\[
	\left\{ \begin{pmatrix}
		x_1 & 0 & 0\\
		0 & \ddots & 0 \\
		0 & 0 & x_r
	\end{pmatrix} \mid x_1 , \ldots , x_r \in \RR \right\}.
	\]
	This is a necessary and sufficient condition for $Q(T) \geq  r$. 
	\bigskip
	
	We formulate the existence of this subspace using the Segre variety: 
	for $\RR$-vector spaces $U,V$ we denote by 
	\[
	\Sigma_{U,V} := \{ [A] \in \mathbb{P} (U \otimes V) \mid \rank
	(A) = 1 \},
	\]
	the projective variety corresponding to the affine cone of rank-one
	tensors of order $2$. For $U=\R ^m$ and $V=\R ^n$ we also just write
	$\Sigma_{m,n}$. 
	By the discussion above, we can now state the following necessary
	condition for subrank $\geq r$.
	\begin{lem}
		\label{lm:geometricupperbound}
		Let $T \in \RR ^{n_1} \otimes \RR ^{n_2} \otimes \RR ^{n_3} $ be a
		tensor with subrank at least $r$. Then there exist linear maps
		$\pi_2: \RR^{n_2} \to \RR^r$, $\pi_3: \RR^{n_3} \to \RR^r$ such
		that the slices $\tilde{T_1}, \ldots , \tilde{T}_{n_1} \in \RR^r
		\otimes \RR^r$ of $(\id \otimes \pi_2 \otimes \pi_3)T$ satisfy:
		\[
		\# (\PP \langle \tilde{T_1}, \ldots , \tilde{T}_{n_1}
		\rangle \cap \Sigma _{r,r} ) \geq r. 
		\]
	\end{lem}
	
	\begin{re}
		The converse of the above is false. For example, the tensor $e_1 \otimes e_1 \otimes e_1 + e_1 \otimes e_2 \otimes e_2  \in \RR ^2 \otimes \RR^2 \otimes \RR ^2$ has slice rank one, and hence also subrank one. However, 
		\[\# (\PP \langle  e_1 \otimes e_1, e_1 \otimes e_2 \rangle \cap
		\Sigma _{2,2} ) = \infty. \qedhere\]
	\end{re}

	\subsection{Typical subranks of \boldmath{$3\times 3\times 5$} tensors}
	
	In this section, we use ideas from \cite{BER-typicalranks}, where the authors use a geometric approach to show that the typical ranks of $3\times 3\times 5$ tensors are 5 and 6. Combining this with the construction in Section~$\ref{sec:geometricmethod}$ gives us the typical ranks of $3\times 3\times 5$ tensors.
	
	Given a sufficiently general tensor $T \in \RR ^3 \otimes \RR ^3
	\otimes \RR ^5$, denote by $T_1, \ldots , T_5 \in \RR ^3 \otimes
	\RR^3$ the slices of this tensor along the third axis. Let $V:=\PP
	\langle T_1, \ldots , T_5 \rangle  \subseteq \PP(\RR^3
	\otimes \RR^3)$ be the projectivization of the space spanned by these
	5 matrices. For a typical tensor, the slices are linearly independent and so $V$ has projective dimension 4. If $Q _\RR (T ) = 3$, there exists a linear subspace $W \subseteq V$ which has the form 
	\[
	\left\{ \begin{pmatrix}
		x & 0 & 0\\
		0 & y & 0 \\
		0 & 0 & z
	\end{pmatrix} \mid x,y,z \in \RR \right\}
	\]
	up to left-right $GL_r \times
	GL_r$ equivalence. 
	
	We thank Tim Seynnaeve for the construction of $T$ in the second part of the following proof.
	
	\begin{thm}
		\label{thm:subrank335}
		The typical subranks of real $3 \times 3 \times 5$ tensors are $2 $ and 3.
	\end{thm}
	\begin{proof}
		The generic subrank is $\min \{3,3,5,\lfloor \sqrt{3 + 3 + 5-2 }\rfloor\} = 3$, so $3$ is a typical subrank. Also, note that 2 is the only typical subrank of $3 \times 3 \times 3$ tensors, so 2 is a lower bound on the possible typical subranks. It is left to show that 2 is a typical subrank of $3 \times 3 \times 5$ tensors.
		
		Using Lemma~\ref{lm:geometricupperbound}, it suffices to show that
		there is an open subset $U \subset \RR^3 \otimes \RR^3 \otimes
		\RR^5$  such that the intersection 
		\[
		\PP \langle T_1, \ldots , T_5 \rangle \cap  \Sigma _{3,3}
		\]
		has at most two (real) points for all $T \in U$, where
		$T_1,\dots,T_5$ are the slices of $T$ along the third axis.
		Indeed, for $T \in U$ we then have $Q(T)<3$.
		
		To show this, consider the tensor $T \in \RR^3 \otimes \RR^3 \otimes \RR^5$ with slices given by 
		\[
		T_1 = \begin{pmatrix}
			1 & 0 & 0\\
			0 & 0 & 0 \\
			0 & 0 & -1 \end{pmatrix},\;\; T_2 = \begin{pmatrix}
			0 & 0 & 0\\
			0 & 1 & 0 \\
			0 & 0 & -1 \end{pmatrix},\;\; T_3 = \begin{pmatrix}
			0 & 1 & 0\\
			1 & 0 & 0 \\
			0 & 0 & 0 \end{pmatrix},\;\; T_4 = \begin{pmatrix}
			0 & 0 & 1\\
			0 & 0 & 0 \\
			1 & 0 & 0 \end{pmatrix},\;\; T_5 = \begin{pmatrix}
			0 & 0 & 0\\
			0 & 0 & 1 \\
			0 & 1 & 0 \end{pmatrix}.
		\]
		The linear space $L:=\langle T_1, \ldots , T_5 \rangle$
		equals 
		\[
		\left\{ \begin{pmatrix}
			x & a & b\\
			a & y & c \\
			b & c & -x-y \end{pmatrix} \mid a,b,c,x,y \in \RR \right\}. 
		\]
		Any matrix $A$ in $L$ is symmetric and real, hence has real
		eigenvalues, which sum up to the trace, namely $0$. This implies
		that if $A$ is nonzero, then it has rank at least $2$. 
		So $\# \PP(L) \cap \Sigma
		_{3,3} = 0$. Being disjoint from $\Sigma_{3,3}$ is an open condition on $L$ (regarded as a point in the Grassmannian of $5$-dimensional subspaces of $3 \times 3$-matrices), hence the conclusion also holds when T is replaced by any nearby tensor.
	\end{proof}
	
	\begin{re}
		For random tensors sampled from a probability distribution equivalent to Lebesgue measure, typical subranks can equivalently be defined as
		$\left\{r \in \mathbb{N} \mid \text{Prob}\{ Q(T) = r\} > 0\right\}$.
		In the context of Theorem~\ref{thm:subrank335}, it was shown by
		\cite{BER-typicalranks} that, with probability one, $\rank (T) = 5$ if
		and only if $\#(\PP \langle T_1,\dots,T_5 \rangle \cap \Sigma_{3,3}) = 6$. There exist four cases with nonzero probability: 
		\[
		\# (\PP\langle T_1,\dots,T_5 \rangle \cap \Sigma_{3,3})  \in
		\{0,2,4,6\}.
		\]
		If the number of intersection points is 6, the rank is 5, as explained by
		\cite{BER-typicalranks}. Thus, there exists a rank decomposition $T = \sum_{i=1}^5 v_{1i} \otimes v_{2i} \otimes v_{3i}$, where $v_{1i}, v_{2i} \in \R^3$ and $v_{3i} \in \R^5$. With probability one, $\{v_{3i}\}_{i=1}^5$ are linearly independent. We define the linear map $\pi_3:\R^5 \to \R^5/\langle v_{34}, v_{35}\rangle$ that projects onto the first three vectors of the decomposition of $T$. Then $(\text{id} \otimes \text{id} \otimes \pi_3)T \in \R^3 \otimes \R^3 \otimes \R^3$ is concise and of rank 3. By Remark \ref{rm:maxsubrank=rank}, $Q((\text{id} \otimes \text{id} \otimes \pi_3)T) = 3$, and hence $Q(T) \geq 3$. Because $T$ was taken to be typical, the subrank must be 3.    
		
		If the number of intersection points is $\leq 2$, the subrank is at most 2.  In particular, if the entries of $T$ are taken to be i.i.d.~Gaussians, the probability that there are at most 2 points in this intersection was shown to be equal to the probability that a random determinantal cubic surface contains 3 or 7 real lines, which is indeed strictly positive. In the remaining case of four intersection points, the argument of Section~\ref{sec:geometricmethod} does not show whether the subrank is 2 or 3.
	\end{re}
	
	\subsection{Typical subranks of \boldmath{$4 \times 4 \times 4$} tensors}
	\label{sec:typSubrank4x4x4}
	In this section, we show that 2 and 3 are both typical subranks of $4
	\times 4 \times 4$ tensors. To do so, we show that the tensor defined
	by the bilinear product of the quaternions has subrank 2. Recall that
	the quaternions $\HH$ are a four dimensional division algebra over 
	$\RR$, with basis $\{1, i, j, k\} $ and multiplication 
	\begin{align*}
		&1x = x1 = x\textrm{ for all }x \in \HH,\\
		&i^2 = j^2 = k^2 = -1,\\
		&ij = k,\; jk = i,\; ki = j, \\
		&ji = -k,\; kj = -i ,\; ik = -j.
	\end{align*}
	The bilinear product $\HH \times \HH \to \HH$ defines a tensor $T \in (\RR ^4 )
	^{\otimes 3}$ whose slices along the first axis are given by
	\[
	T_1 = \begin{pmatrix}
		1 & 0 & 0 & 0 \\
		0 & 1 & 0 & 0 \\
		0 & 0 & 1 & 0 \\
		0 & 0 & 0 & 1
	\end{pmatrix}, \;\;\;\;
	T_2 = \begin{pmatrix}
		0 & -1 & 0 & 0 \\
		1 & 0 & 0 & 0 \\
		0 & 0 & 0 & -1 \\
		0 & 0 & 1 & 0
	\end{pmatrix}, \;\;\;\;
	T_3 = \begin{pmatrix}
		0 & 0 & -1 & 0 \\
		0 & 0 & 0 & 1 \\
		1 & 0 & 0 & 0 \\
		0 & -1 & 0 & 0
	\end{pmatrix}, \;\;\;\;
	T_4 = \begin{pmatrix}
		0 & 0 & 0 & -1 \\
		0 & 0 & -1 & 0 \\
		0 & 1 & 0 & 0 \\
		1 & 0 & 0 & 0
	\end{pmatrix}.
	\]
	
	Let $a \in \HH$; then left multiplication defines a linear map 
	\[
	L_a : \HH \to \HH , \; b \mapsto a \cdot b,
	\]
	where $a \cdot b$ denotes the image of the bilinear product. 
	For $a= c_1 1 + c_2 i + c_3 j + c_4 k \in \HH$, the matrix of $L_a$ is
	$c_1 T_1+ c_2 T_2 + c_3 T_3 + c_4 T_4$. Since the quaternions are a
	division algebra, this matrix has full rank $4$ for all $a \in \HH \setminus \{ 0 \}$. 
	\begin{lem}
		\label{lm:quaternions}
		The real subrank of the tensor $T \in (\RR^4)^{\otimes 3}$ defined by the quaternions is 2 and there is a neighborhood of $T$ where the real subrank is 2.
	\end{lem}
	
	\begin{proof}
		We start by showing that $Q(T ) \leq 2$. Let $T_1,\dots,T_4$ be the
		slices of $T$ along the first axis. As described above, all nonzero
		matrices in the linear space $\langle T_1 , \ldots , T_4\rangle
		\subseteq \RR^4 \otimes \RR ^4 $ are invertible.  Applying two linear
		maps $\pi_2,\pi_3:\RR^4 \to \RR^3$ to a rank-four matrix can reduce
		the rank at most by $2$, so that it still has rank $>1$. Hence by
		Lemma~\ref{lm:geometricupperbound}, $Q(T) \leq 2$. 
		
		Next, we show that there is a neighborhood in the Euclidean topology of $T$ where all tensors have real subrank at most 2. For all $v, w \in \RR ^4$ with $\Vert v \Vert = \Vert w \Vert =1$, there exist $\varepsilon, \delta >0 $ with the following property. For all $S \in B_\delta (T) $ and all $\tilde{v} \in B_\varepsilon (v) \subseteq S^3, \tilde{w} \in B_\varepsilon (w) \subseteq S^3$, the intersection
		\[
		\PP \langle \tilde{S}_1, \ldots , \tilde{S}_4 \rangle \cap \Sigma _{\tilde{v}^\perp , \tilde{w} ^\perp }
		\]
		is empty, where $\tilde{S}_1, \ldots , \tilde{S}_4 $ are the four
		slices of $\tilde{S} := (\id \otimes \pi _{\tilde{v}^\perp}
		\otimes \pi _{\tilde{w}^\perp} )(S) $ along the first axis, $\pi
		_{\tilde{v}^\perp} : \RR^4 \to \tilde{v}^\perp$ is the orthogonal
		projection 
		and $\pi _{\tilde{w}^\perp} $ is defined analogously. The open
		balls $B_\varepsilon (v) \times B_\varepsilon (w)$ form an open
		cover of $S^3 \times S^3$, so using compactness, there exists a
		finite subcover. Let $\delta_0$ be the minimal $\delta $ of these
		coverings. Then for all $S \in B_{\delta_0} (T) $, we have  
		\[
		\#  \PP \langle \tilde{S}_1, \ldots , \tilde{S}_4 \rangle \cap \Sigma _{v^\perp , w ^\perp } =0 
		\]
		for all $v,w \in S^3 $, which shows that $Q (S) \leq 2$.
		
		For completeness, to show that the subrank of $T$ is 2, one can check that $(A \otimes B \otimes C)T = I_2 $ for
		\[
		A = \begin{pmatrix}
			1 & 0 & 0 & 0 \\
			0 & 1 & 0 & 0 \\
		\end{pmatrix}, \;\;\; 
		B = \begin{pmatrix}
			1 & 0 & 0 & 0 \\
			0 & 0 & 1 & 0 \\
		\end{pmatrix}, \;\;\;
		C = \begin{pmatrix}
			1 & 0 & 0 & 0 \\
			0 & 0 & 0 & 1 \\
		\end{pmatrix} .
		\]
		
		Lastly, the fact that $(A \otimes B \otimes C)T = I_2 $ implies that the Cayley hyperdeterminant $\Delta ((A \otimes B \otimes C)T_\varepsilon ) >0$ for $T_\varepsilon $ in a small enough neighborhood of $T$. This shows that in a neighborhood of $T$, all tensors have subrank 2.    
	\end{proof}
	
	\begin{thm}
		For $ 4 \times 4 \times 4$ tensors, the typical subranks are 2 and 3.
	\end{thm}
	\begin{proof}
		We know that $3$ is the largest typical subrank, as it is the generic subrank for tenors in $\CC ^4 \otimes  \CC ^4 \otimes \CC ^4$. As 2 is the typical subrank of $2 \times 2 \times 3$ tensors, we have that the typical subranks are at least 2. So it is left to show that $2$ is a typical subrank. In Lemma~\ref{lm:quaternions}, we show that the tensor of the quaternions has a neighborhood with subrank 2. Hence, 2 is a typical subrank.
	\end{proof}
	
	\begin{cor}
		For real tensors of the format $3 \times 4 \times 4$, the typical subranks are 2 and 3.
	\end{cor}
	\begin{proof}
		We know that $3 = \lfloor \sqrt{3 + 4 + 4 -2} \rfloor$ is a typical subrank. If 2 were not typical, this would imply that 2 is also not a typical subrank for tensors in $\RR ^4 \otimes \RR ^ 4\otimes \RR^4 $.
	\end{proof}
	
	\begin{re}
		The same argument shows that the tensor defined by the complex
		numbers, seen as a two dimensional real algebra, has subrank 1.
		For the octonions, one can constuct projections onto
		five-dimensional subspaces and show that the intersection with the
		Segre variety is still empty. This shows that the subrank of the
		octonions is at most 4, which equals the generic subrank of $8
		\times 8 \times 8$ tensors. This also follows from Theorem~\ref{thm:upperBoundQDivision} below. However, we conjecture something stronger.
	\end{re}
	
	\begin{conjecture} \label{conj:Octonions}
		The real subrank of the $8 \times 8 \times 8$ multiplication tensor of
		the octonions equals $3$.
	\end{conjecture}
	
	\section{Subrank of direct sums of real division algebras}
	
	In this section, we examine the subrank of multiplication tensors
	of division algebras. Let $D$ be a real division algebra, i.e.,
	a finite-dimensional vector space equipped with a bilinear map $D
	\times D \to D,\ (a,b) \mapsto ab$ such that for each nonzero $a$
	the left multiplication $L_a: b \mapsto ab$ is invertible. No further
	conditions are imposed; in particular, the multiplication may be
	non-associative or even non-alternative. Of course, by the celebrated
	Hopf-Bott-Milnor-Kervaire theorem \cite{Hopf41,Bott58,Kervaire58}, it
	follows that $\dim_\RR D \in \{1,2,4,8\}$. The bilinear map $f:D \times
	D \to D$ has subrank 1 if $D = \CC$ (see Example~\ref{ex:subrankC})
	and 2 if $D = \HH$ (see Section~\ref{sec:typSubrank4x4x4}).

	We now turn to direct sums of the tensors corresponding to multiplication in those division algebras; i.e.~componentwise multiplication over $D^n$.

	\begin{prop} \label{prop:Subrank}
		Let $\RR^{n_1},\RR^{n_2},\RR^{n_3}$ be finite-dimensional $\RR$-vector spaces, and let 
		\[ f:\RR^{n_1} \times \RR^{n_2} \to \RR^{n_3},\ (u,v) \mapsto u*v \] 
		be a bilinear map. Then $Q(f)$ is the maximal
		$r$ for which there exist $u_1,\ldots,u_r \in \RR^{n_1}$ and $v_1,\ldots,v_r
		\in \RR^{n_2}$ such that $u_1*v_1,\ldots,u_r*v_r$ are linearly independent
		modulo the space 
		\[ \langle \{u_i*v_j \mid 1 \leq i,j \leq r, i \neq j\} \rangle_\RR. \] 
	\end{prop}
	
	\begin{proof}
		Set $r:=Q(f)$. Then there exist linear maps $\varphi_1:\RR^r \to \RR^{n_1}$,
		$\varphi_2:\RR^r \to \RR^{n_2}$, and $\varphi_3:\RR^{n_3} \to \RR^r$ such that the composition
		\[ \varphi_3 \circ f \circ (\varphi_1 \times \varphi_2) = I_r:\RR^r \times \RR^r \to
		\RR^r \] is the component-wise multiplication. Denote the standard basis
		of $\RR^r$ by $e_1,\ldots,e_r$ and set $u_i:=\varphi_1(e_i)$
		and $v_j:=\varphi_2(e_j)$ for $i,j=1,\ldots,r$. Then $\varphi_3(u_i*v_i)=e_i$
		and $\varphi_3(u_i*v_j)=0$ for $i \neq j$. Hence the $u_i \in \RR^{n_1}$ and the
		$v_j \in \RR^{n_2}$
		have the required property.
		
		Conversely, if the $u_i$ and the $v_i$ have the property in the
		proposition, then define $\varphi_1(e_i):=u_i$, $\varphi_2(e_j):=v_j$, and note
		that there exists a linear map $\varphi_3:\RR^{n_3} \to \RR^r$ that maps $u_i*v_i$ to
		$e_i$ and all $u_i*v_j$ with $i \neq j$ to zero. Then $\varphi_3 \circ f
		\circ (\varphi_1 \times \varphi_2) = I_r:\RR^r \times \RR^r \to \RR^r$ is the componentwise scalar multiplication, so $Q(f) \geq r$.
	\end{proof}
	
	\begin{lm} \label{lm:Characterisation}
		For a bilinear map $f:\RR^{n_1} \times \RR^{n_2} \to \RR^{n_3},\ (u,v) \mapsto u*v$ and a $q
		\in \ZZ_{\geq 0}$ the following two statements are equivalent:
		\begin{enumerate}
			\item $Q(f) \leq q$; and 
			\item given any integer $r \geq q$ and any $u_1,\ldots,u_r \in \RR^{n_1}$ and
			$v_1,\ldots,v_r \in \RR^{n_2}$, define the space of linear relations
			\[ R:=\{\alpha \in \RR^{[r] \times [r]} \mid \sum_{i,j}
			\alpha_{ij}u_i*v_j=0\} \]
			and the projection to the diagonal
			\[ \pi:\RR^{[r] \times[r]} \to \RR^r,\ \alpha \mapsto
			(\alpha_{11},\ldots,\alpha_{rr}). \]
			Then $\dim \pi(R) \geq r-q$.
		\end{enumerate}
	\end{lm}
	
	\begin{proof}
		Suppose that the second statement holds; let $r>q$, and let
		$u_1,\ldots,u_r \in \RR^{n_1}$ and $v_1,\ldots,v_r \in \RR^{n_2}$. By the second
		statement, the $u_i*v_i$ are not linearly independent modulo the $u_i*v_j$
		with $i \neq j$. Hence by Proposition~\ref{prop:Subrank}, $Q(f) \leq q$.
		
		Now suppose that $Q(f) \leq q$ and let $r \geq q$ and $u_1,\ldots,u_r \in
		\RR^{n_1}$ and $v_1,\ldots,v_r \in \RR^{n_2}$. Let $s$ be the dimension of the image of $\langle u_1*v_1,
		\ldots, u_r*v_r \rangle_\RR$ under projection modulo the space $W$ spanned by
		the $u_i*v_j$ with $i \neq j$. By Proposition~\ref{prop:Subrank}, $s \leq
		Q(f)$, and hence $s \leq q$. Without loss of generality, $u_1*v_1,\ldots,u_s*v_s$ are
		linearly independent modulo $W$. Then $u_1*v_1,\ldots,u_s*v_s,u_i*v_i$
		are not, for any $i=s+1,\ldots,r$.  This gives rise to $r-s \geq r-q$
		elements in $R$ whose images under $\pi$ are linearly independent.
	\end{proof}
	
	We are now ready to prove the main result of this section.
	
	\begin{thm}\label{thm:upperBoundQDivision}
		Let $D$ be a division algebra over $\RR$ of dimension $\geq 2$. Let $f_n$ be the componentwise multiplication map  
		\[
		f_n: D^n \times D^n \to D^n; (a,b)=((a_1, \ldots a_n),(b_1, \ldots , b_n)) \mapsto ( a_1  b_1, \ldots , a_n  b_n)=: a* b.
		\]
		Regarding $f_n$ as an $\RR$-bilinear map, we have $Q(f_n) \leq nd$,
		where $d = \frac{1}{2} \dim _\RR D$.
	\end{thm}
	
	\begin{proof}
		We proceed by induction on $n$. For $n=0$ the statement is obvious;
		so we assume that $n>0$ and that the statement is true for all
		strictly smaller values of $n$. 
		
		By Proposition~\ref{prop:Subrank}, we need to show that if $u_1,\ldots,u_r
		\in D^n$ and $v_1,\ldots,v_r \in D^n$ have the property that the $u_i*v_i$
		are $\RR$-linearly independent modulo the space $W$ spanned over $\RR$
		by the $u_i*v_j$ with $i \neq j$, then $r \leq n d$.
		
		It follows immediately that $u_1,\ldots,u_r$ are linearly independent
		over $\RR$, and so are $v_1,\ldots,v_r$: indeed, if $u_i=\sum_{j \neq i}
		\gamma_j u_j$ for certain $\gamma_j \in \RR$, then $u_i*v_i$ lies in $W$,
		a contradiction to the assumption that the $u_i*v_i$ are linearly
		independent modulo $W$. 
		
		Suppose, for a contradiction, that $r>nd$. After permuting coordinates,
		we may assume that $u_r \in \{0\}^{n-k} \times (D\setminus\{0\})^k $ with $k>0$.
		
		We define the left multiplication map by $u_r$ as 
		\[
		L_{u_r} : D^n \to \{0\}^{n-k} \times D^k, \quad   v \mapsto  u_r * v.
		\]
		We restrict this map to $\langle v_1,\ldots,v_r \rangle_\RR$ and define
		\[ r_1:=\dim_\RR L_{u_r}( \langle \{ v_j \mid j=1,\ldots,r \} \rangle_\RR ) = \dim_\RR \langle \{ u_r*v_j \mid j=1,\ldots,r \} \rangle_\RR. \]
		Let $r_2 = r -r_1$ be the dimension of the kernel of said restriction
		of $L_{u_r}$. Since $D$ is a division algebra, 
		\[ r_2=\dim_\RR \langle v_1,\ldots,v_r \rangle_\RR \cap (
		D^{n-k} \times \{0\}^k ). \]
		
		Define 
		\[ V:=\{\gamma \in \RR^r \mid \sum_{i=1}^r \gamma_i v_i \in 
		D^{n-k} \times \{0\}^k \}, \]
		so that $\dim_\RR(V)=r_2$. Let $I \subseteq [r]$ be a subset of
		cardinality $r_2$ such that the projection $\RR^r \to \RR^I$ restricts
		to an isomorphism on $V$. Then for all $j \in I$ we have
		\[ \widetilde{v}_j:=v_j+\sum_{l \not \in I} \gamma_{jl} v_l \in  D^{n-k} \times \{0\}^k \]
		for suitable coefficients $\gamma_{jl} \in \RR$. We claim that the vectors
		$u_i*\widetilde{v}_i$ with $i \in I$ are linearly independent modulo
		the space $\widetilde{W}$ spanned by the $u_i*\widetilde{v}_j$ with $i,j
		\in I$ and $i \neq j$. Indeed, suppose that 
		\[ \sum_{i,j \in I} \alpha_{ij} u_i*\widetilde{v}_j = 0 \]
		for certain scalars $\alpha_{ij} \in \RR$, where the $\alpha_{ii}$ are
		not all zero. Substituting the $\widetilde{v}_j$, we obtain 
		\[ 0=\sum_{i,j \in I} \alpha_{ij} (u_i*v_j + \sum_{l \not \in I}
		\gamma_{jl} u_i*v_l) = \sum_{i,j \in I} \alpha_{ij} u_i*v_j 
		+ \sum_{i \in I} \sum_{l \not \in I} (\sum_{j \in I} \alpha_{ij}
		\gamma_{jl}) u_i*v_l, \]
		which is a linear relation among the $u_i*v_j$ with $i\in I, j \in
		[r]$ in which the coefficients of the $u_i*v_i$ are not all zero,
		a contradiction. This proves the claim. Note further that the claim
		remains intact if we replace, in each $u_i$ with $i \in I$, the last
		$k$ entries by $0$s---after all, this does not affect the products
		$u_i*\widetilde{v}_j$. Hence, by the induction hypothesis applied to $n-k$,
		we find that $r_2 \leq (n-k)d$, and therefore $r_1=r-r_2\geq r-(n-k)d>kd$,
		where the last inequality follows from the assumption that $r>nd$.
		
		Next let $\varphi:D^n \to D^{n-k}$ be the projection onto the 
		first $n-k$ coordinates. Define 
		\[ u'_i:=\varphi(u_i) \text{ and } v'_j:=\varphi(v_j)
		\text{ for } i,j=1,\ldots r-1.\]
		Let $R$ be the space of relations among their products:
		\[ R:=\{\alpha \in \RR^{[r-1] \times [r-1]} \mid \sum_{i,j} \alpha_{ij}
		u'_i*v'_j=0\},  \]
		and let $\pi:\RR^{[r-1] \times [r-1]} \to \RR^{r-1}$ be the projection
		to the diagonal. By Lemma~\ref{lm:Characterisation} and the induction
		hypothesis applied to these vectors, the projection $\pi(R)$ has dimension
		at least $(r-1)-(n-k)d\geq kd$ (here we use that $r-1 \geq nd$, which
		follows from $r>nd$ and the fact that $d$ is an integer, i.e., that
		$\dim(D) \geq 2$). Let $R' \subseteq R$ be a $kd$-dimensional subspace
		on which $\pi$ is injective.
		
		Now consider the $\RR$-linear map $\psi:R' \to  \{0\}^{n-k} \times D^k
		\subseteq D^n$ defined by
		\[ \psi(\alpha)=\sum_{i,j} \alpha_{ij} u_i*v_j. \]
		Any nonzero $\alpha$ in the kernel of $\psi$ would give a linear relation
		among the $u_i*v_j$ with $i,j=1,\ldots,r-1$ in which the coefficients
		of the $u_i*v_i$ are not all zero, a contradiction. Hence $\psi$ is
		injective. Therefore $\dim_\RR(\im(\psi))=kd$. But now $\im(\psi)$ and
		$\langle u_r*v_1,\ldots,u_r*v_r \rangle_\RR$ are $\RR$-subspaces of $\{0\}^{n-k} \times D^k$ of dimensions $kd$ and $r_1 >kd$, respectively. Since
		$d=\frac{1}{2}\dim_\RR (D)$, we have $\dim_\RR ( \{0\}^{n-k} \times D^k )=2kd$, so these
		spaces intersect nontrivially: say
		\[ \psi(\alpha)=\beta_1 u_r*v_1 + \cdots + \beta_r u_r*v_r \neq 0 \]
		for suitable $\alpha \in R'$ and $\beta_1,\ldots,\beta_r \in \RR$.
		But then 
		\[ \sum_{i,j \in [r-1]} \alpha_{ij} u_i*v_j 
		- \sum_{j=1}^r \beta_j u_r*v_j = 0 \]
		and not all $\alpha_{ii}$ are zero, a contradiction. We conclude that
		$r \leq nd$, as desired. 
	\end{proof}
	
	\begin{cor}
		\label{cor:subrankCandH}
		Let $D$ be the division algebra $\CC $ or $\HH$  over $\RR$. Let $f_n$ be the componentwise multiplication map  
		\[
		f_n: D^n \times D^n \to D^n; (a,b)=((a_1, \ldots a_n),(b_1, \ldots , b_n)) \mapsto ( a_1  b_1, \ldots , a_n  b_n)=: a* b.
		\]
		Regarding $f_n$ as an $\RR$-bilinear map, the subrank is $Q(f_n) = nq$, where $q$ is the subrank of $f_1$. In other words, the subrank is additive for direct sums of the division algebra with itself.
	\end{cor}
	\begin{proof}
		The direction $Q(f_n) \geq nq $ always holds for direct sums of
		tensors. The other direction follows from
		Theorem~\ref{thm:upperBoundQDivision} above using that for $D =
		\CC $ or $D = \HH$ we have $Q(f_1 ) = \dim_{\RR} (D) /2$.
	\end{proof}
	\begin{re}
		For the octonions, the above upper bound is also valid. However,
		we do not know whether $Q(f_1) = 3$ or 4; see
		Conjecture~\ref{conj:Octonions}. 
	\end{re}
	
	\section*{Competing interests}
	The authors have no competing interests to declare.
	
	\bibliographystyle{alpha}\bibliography{math}

\end{document}